\documentclass[12pt,a4paper,oneside]{article}
\usepackage[T1]{fontenc}
\usepackage[english]{babel}
\usepackage{amsmath}
\usepackage{amsfonts}
 \usepackage{amssymb}
 \usepackage{amsthm}
\usepackage{graphicx} 
\usepackage{float} 

\usepackage{hyperref}

\theoremstyle{definition}
\newtheorem{theorem}{Theorem} 

\theoremstyle{definition}
\newtheorem{lemma}{Lemma}

\theoremstyle{definition}

\theoremstyle{definition}
\newtheorem{ex}{Example}  

\theoremstyle{definition}
\newtheorem{note}{Note} 
 
\title{Global dynamics of a periodic SEIRS model with general incidence rate}

\author{ Eric \'Avila-Vales, Erika Rivero-Esquivel and Gerardo Garc\'ia-Almeida \\
 \begin{small}
Facultad de Matem\'aticas, Universidad Aut\'onoma de Yucat\'an.
\end{small}  \\
\begin{small}
Anillo Perif\'erico Norte, Tablaje 13615, C.P. 97119.
M\'erida, M\'exico.
\end{small}}

\date{}
 
\begin{document}
\maketitle

\begin{abstract}
We consider a family of periodic SEIRS epidemic models with a fairly general incidence rate and it is shown the basic reproduction number determines the global dynamics of the models and it is a threshold parameter for persistence. Numerical simulations are performed to estimate the basic reproduction number and  illustrate our analytical findings, using a nonlinear incidence rate.
\end{abstract}

\section{Introduction} 
Epidemiological models have been recognized as valuable tools in analyzing the spread and control of infectious diseases. In the study of epidemiological models, incidence rate plays an important role.  An incidence rate is defined as the number of new health related events or cases of a disease in a population exposed to the risk in a given time period. Incidence rate has been developed by many authors. In order to model this disease transmission process several authors employ the incidence functions: The earliest one is the bilinear incidence rate $\beta SI$ used by Kermack and Mckendrick \cite{kermack1927contribution} in 1927, where $\beta, S$ and $I$ denote the transmission rate, the number of susceptible population and the infectious population respectively. It is based on the law of mass action  which is not realistic. So there is a need to modify the classical linear incidence rate to study the dynamics of infection among large population. In 1978, Capasso and Serio \cite{capasso1978generalization} introduced a saturated incidence rate by research of the Cholera epidemic spread in Bari. Also in 1978, May and Anderson \cite{anderson1978regulation} proposed the saturated incidence rate. \par 
 
In the present work we focus on SEIRS epidemic models. We improve the model of Moneim and Greenhalgh in \cite{moneim2005use}, introducing an incidence rate with a general function taken from \cite{bai2015threshold} and the references therein. \par 
We propose the following SEIRS model:

\begin{equation} \label{ec1}
\begin{aligned}
\frac{dS}{dt} &= \mu N(1-p) - \beta(t)Sf(I)-( \mu+r(t) )S + \delta R \\ 
\frac{dE}{dt} &= \beta(t) S f(I)- (\mu + \sigma )E \\
\frac{dI}{dt} &= \sigma E - (\mu + \gamma )I \\
\frac{dR}{dt} &= \mu N p+ r(t)S+ \gamma I - (\mu + \delta) R.
\end{aligned}
\end{equation}

Where $N=S+E+I+R$ is the total population size, with $S,E,I,R$ denoting the fractions of population that are suceptible, exposed, infected and revovered, respectively. $ \beta(t) $ is the transmission rate and it is a continuous, positive T -periodic function. $p$ ($0\leq p \leq 1 $) is the vaccination rate of all new-born children. $r(t)$ is the vaccination rate of all susceptibles in the population and it is a continuous, positive periodic function with period $LT$ , where $L$ is an integer. $\mu$ is the common per capita birth and death rate. $\sigma, \gamma$ and $ \delta $ are the per capita rates of leaving the latent stage, infected stage and recovered stage, respectively. It is assumed that parameters are positive constants. \par 
Bai and Zhou in \cite{bai2012global} answered some open problems stated in \cite{moneim2005use} , they also shown that their condition is a treshold between persistence and extinction of the disease via the framework established in \cite{wang2008threshold}. They assumed that the incidence was bilinear. In our study, the nonlinear assumptions on function $f$ are listed below (see \cite{bai2015threshold} ):
\begin{itemize}
\item[A1)] $f: \mathbb{R}_+ \rightarrow \mathbb{R}_+ $ is continuously differentiable. 
\item[A2)] $f(0)=0, f'(0)>0$ and $f(I)>0$ for all $I>0$.
\item[A3)] $f(I)-If'(I) \geq 0 $.
\end{itemize}
 Under these assuptions, function $f(I)$ includes various types of incidence rate, particularly, when $f(I)=I$, we are on the bilinear case considered by Moneim. \par  In addition, we assume following extra conditions (see \cite{posny2014modelling}):
\begin{itemize}
\item[A4)] $f''(0)\leq 0$.
\item[A5)] There exists $ \epsilon^{*}>0 $ such that when $0<I< \epsilon^{*}$, $f(I) \geq f(0)+I f'(0)+ \frac{1}{2} I^{2} f''(0). $
\end{itemize}

This set of assumptions on the function f allows for more general incidence functions than the bilinear one, like saturated incidence functions and functions of the form $\beta SI/(1+kI^q)$, in particular in the case when $q>1$, they represent psychological or media effects depending on the infected population. In this last case the incidence function is non monotone on I. A3) regulates the value of $f(I)$ comparing it with the value at $I$ of a line containing the origin of slope $f'(I)$ (Note that this line varies as I increases), A4) requires a concave $f(I)$ at the origin, and A5) imposes the geometrical condition that in a small neighborhood of the origin $f(I)$ must lie between the tangent line of f at I and a concave parabola tangent to f at I. \par 
  The layout of this paper is as follows.In section 2,we
introduce the basic reproduction number via the theory developed in \cite{bacaer2006epidemic}, \cite{wang2008threshold}. In section 3, we adapt the arguments given in \cite{bai2012global} to prove that the disease free periodic solution is globally asymptotically stable if $ \mathcal{R}_0 <1$ and that if $\mathcal{R}_0>1$ system \eqref{ec1} is persistent.\par 
We consider a family of $SEIR$ models with periodic coefficients with general incidence rate in epidemiology. We show that the global dynamics is determined by the basic reproduction number $\mathcal{R}_0$ Our results generalize the ones in \cite{bai2012global}.

\section{The basic reproduction number}

First of all, we prove non-negativity of solutions under non-negative initial conditions. 

\begin{theorem}
Let  $S_0,E_0,I_0,R_0 \geq 0$, the solution of \eqref{ec1} with $$(S(0),E(0),I(0),R(0))=(S_0,E_0,I_0,R_0)$$ is non negative in sense that $S(t),E(t),I(t),R(t) \geq 0$ , $\forall t>0$, and satisfies $S(t)+E(t)+I(t)+R(t) = N$, with $N$ constant. \label{teo0} 
\end{theorem}
\begin{proof}
Let $x(t)=(S(t),E(t),I(t),R(t))$ be the solution of system \eqref{ec2} under initial conditions $x_0=(S(0),E(0),I(0),R(0))=(S_0,E_0,I_0,R_0) \geq 0$, by continuity of solution, for all of $S(t),E(t),I(t)$ and $R(t)$ that have a positive initial value at $t=0$, we have the existence of an interval $(0,t_0)$ such that $S(t),E(t),I(t),R(t) \geq 0$ for $0<t<t_0$. We will prove that $t_0= \infty$. \par 
If $S(t_1)=0$ for a $t_1\geq 0$ and other components remain non-negative at $t=t_1$, then 
$$ \frac{dS}{dt}(t=t_1)= \mu N(1-p)+ \delta R \geq 0,$$
this implies that whenever the solution $x(t)$ touches the $S$-axis, the derivative of $S$ is non decreasing and the function $S(t)$ does not cross to negative values. Similarly: 
When $E(t_1)=0$ for a $t_1>0$ and other components remain non-negative:
$$ \frac{dE}{dt} (t=t_1) = \beta(t)S f(I) \geq 0. $$
When $I(t_1)=0$ for a $t_1>0$ and other components remain non-negative:
$$ \frac{dI}{dt} (t=t_1) = \sigma E \geq 0. $$
Finally, when $R(t_1)=0$ for a $t_1>0$ and other components remain non-negative:
$$ \frac{dR}{dt} (t=t_1) = \mu Np+r(t)S+ \gamma I \geq 0. $$
Therefore, whenever $x(t)$ touches any of the axis $S=0,E=0,I=0,R=0$, it never crosses them. \par 
Now, let $N(t)=S(t)+E(t)+I(t)+R(t)$, then adding all equations of system we can see that $\frac{dN}{dt} =0$, so the value of $N$ is constant.
\end{proof}
To reduce the system \eqref{ec1}, let  With $R=N-S-E-I$,  system \eqref{ec1} is reduced to 

\begin{equation} \label{ec2}
\begin{aligned}
\frac{dS}{dt} &= \mu N(1-p)- \beta(t)S f(I)-(\mu + r(t))S+ \delta(N-S-E-I), \\
\frac{dE}{dt} &= \beta(t)Sf(I)- ( \mu + \sigma )E, \\
\frac{dI}{dt} &= \sigma E - ( \mu + \gamma )I.
\end{aligned}
\end{equation}
 
The dynamics of system \eqref{ec1} is equivalent to that of \eqref{ec2}. Moreover, due to positivity of solutions $S+E+I \leq N$, so we study the dynamic of system \eqref{ec2}  in  the region 
\begin{equation}
X= \{ (S,E,I) \in \mathbb{R}_+^{3}: S+E+I \leq N \} \label{ec10}.
\end{equation} 
 A  disease free periodic  solution can be found for \eqref{ec2}. To find it, set $E=0=I$, then from first equation of \eqref{ec2}:

\begin{equation}
\dfrac{dS}{dt} = \mu N(1-p)-( \mu+r(t) ) S+ \delta (N-S), \quad S(0)= S^{0} \in \mathbb{R}_+. \label{ec12}
\end{equation} 
From \cite{bai2012global} and \cite{moneim2005use}, the equation above admits a unique positive LT-periodic solution given by:
\begin{equation} \label{ec19}
\hat{S}(t) = e^{-\int_0^{t}( \mu+r(s)+ \delta )ds} \left( \hat{S}(0)+ N( \mu(1-p)+ \delta ) \int_0^{t} e^{\int_0^{s} (\mu+r( \xi ) + \delta) d \xi}ds \right),
\end{equation}
where $$ \hat{S}(0) = \dfrac{N( \mu(1-p)+ \delta ) \int_0^{LT} e^{\int_0^{s} ( \mu+r(\xi)+ \delta )d \xi} ds }{ e^{\int_0^{LT}(\mu+r(s)+\delta)ds} -1} .$$
Then, $(\hat{S}(t),0,0)$ is a periodic infection free solution of \eqref{ec2}, moreover, from \cite{bai2012global} we have that $\hat{S}(t)<N$, therefore, $( \hat{S}(t),0,0 ) $ lives in $X$. \par 
Using the notation of \cite{van2002reproduction}, we sort the compartments so that the first 2 compartments correspond to infected individuals. Let $x=(E,I,S)$ and define
\begin{itemize}
\item  $\mathcal{F}_i$: the rate of new infection in compartment i.
\item $\mathcal{V}_i^{+}: $ the rate of individuals into compartment i by other means.
\item $\mathcal{V}_i^{-}:$ the rate of transfer individuals out of compartment i.
\end{itemize}

 System can be written as 
\begin{align}
 x'(t) &= \left(
 \begin{matrix}
 \beta(t)Sf(I)- ( \mu + \sigma )E \\
 \sigma E - ( \mu + \gamma )I \\
 \mu N(1-p)- \beta(t)S f(I)-(\mu + r(t))S+ \delta(N-S-E-I)
\end{matrix}  \right) \nonumber \\
 &= \mathcal{F}- \mathcal{V}, \label{ec3}
\end{align}
where  $ \mathcal{V}= \mathcal{V}^{-}- \mathcal{V}^{+} $,
\begin{align}
\mathcal{F} &= \left( \begin{matrix}
\beta(t) S f(I) \\ 0\\0
\end{matrix} \right), \quad
\mathcal{V}^{+} = \left( \begin{matrix}
0 \\ \sigma E \\ \mu N(1-p) + \delta N
\end{matrix} \right), \nonumber \\
\mathcal{V}^{-}& = \left( \begin{matrix}
(\mu+ \sigma)E \\ (\mu + \gamma)I \\ \beta(t)Sf(I)+ \delta(S+E+I)+(\mu+r(t))S
\end{matrix} \right).
\end{align}

Linearizing system \eqref{ec3} around the disease free solution, we obtain the matrix of partial derivatives $J(0,0,\hat{S}) = D \mathcal{F}(0,0,\hat{S})- D \mathcal{V}(0,0,\hat{S})$, where 

\begin{align}
D \mathcal{F}(0,0,\hat{S}) &= \left( \begin{matrix}
0 & \beta(t) \hat{S} f'(0) & 0 \\ 0 & 0 & 0 \\ 0& 0& 0
\end{matrix} \right) \\
D \mathcal{V} ( 0,0,\hat{S} ) &= \left( 
\begin{matrix}
\mu + \sigma & 0 & 0 \\ - \sigma & \mu+ \gamma & 0 \\ \delta & \beta(t) \hat{S} f'(0)+ \delta &  \delta + \mu + r(t)
\end{matrix}
\right).
\end{align}
Using lemma 1 of \cite{van2002reproduction}, we part $ D \mathcal{F} $ and $ D \mathcal{V} $  and set 
\begin{equation}
F(t) = \left( \begin{matrix}
0 & \beta(t) \hat{S} f'(0) \\ 0 & 0 
\end{matrix} \right), \quad
V(t) = \left( 
\begin{matrix}
\mu + \sigma & 0 \\ - \sigma & \mu+ \gamma 
\end{matrix}
\right). \label{ec6}
\end{equation}
For a compartmental epidemiological model based on an autonomous system, the basic reproduction
number is determined by the spectral radius of the next-generation matrix $FV^{-1}$(which is
independent of time) \cite{van2002reproduction}. The definition of basic reproduction number for non autonomous systems has been studied for multiple authors, see for example \cite{bacaer2006epidemic} and \cite{wang2008threshold}. Particularly, Wang and Zhao in \cite{wang2008threshold} extended the work of \cite{van2002reproduction} to include epidemiological models in periodic environments. They introduced the next infection operator  $\mathcal{L}: C_{LT} \rightarrow C_{LT} $ given by

\begin{equation}
( \mathcal{L} \phi )(t) = \int_0^{\infty} Y(t,t-a)F(t-a)\phi(t-a)da, \quad \forall t \in \mathbb{R}, \phi \in C_{LT},
\end{equation}
  where $C_{LT}$ is the ordered Banach space of all $LT$ periodic functions form $ \mathbb{R} $ to $ \mathbb{R}^{2} $, which is equipped with the maximum norm. $ \phi(s) \in C_{LT} $ is the initial distribution of infectious individuals in this periodic environment, and $Y(t,s)$, $t \geq s$ is the evolution operator of the linear periodic system:
  \begin{equation}
  \dfrac{dy}{dt} = -V(t) y,
    \end{equation}
 that means, for each $ s \in \mathbb{R}$, the 2x2 matrix $Y$ satisfies
 \begin{equation}
 \dfrac{dY(t,s)}{dt} = -V(t)Y(t,s), \quad \forall t \geq s, Y(s,s)=I_{2\times 2}.
\end{equation}    

$ \mathcal{L} \phi $ is the distribution of accumulative new infections at time $t$ produced by all those infected individuals $ \phi(s) $ introduced before $t$, with kernel $K(t,a)= Y(t,t-a)F(t-a) $. The coefficient $K_{i,j}(t,a)$ in row $i$ and column $j$ represents the expected number of individuals in compartment $I_i$ that one individual in compartment $I_j$ generates at the beginning of an epidemic per unit time at time $t$ if it has been in compartment $I_j$ for $a$ units of time, with $I_1=E, I_2=I$ \cite{bacaer2007approximation}.\par  
Let $r_0>0,$ $r_0$ is an eigenvalue of $\mathcal{L}$ if there is a nonnegative eigenfunction $v(t) \in C_{LT}$ such that

\begin{equation}
\mathcal{L} v= r_0 v.
\end{equation}

  Therefore, the basic reproduction number is defined as 
\begin{equation}
\mathcal{R}_0 := \rho( \mathcal{L} ), \label{ec4}
\end{equation}
 the spectral radius of $ \mathcal{L}$. The basic reproduction number can be evaluated by several numerical methods and approximations (\cite{bacaer2007approximation}, \cite{mitchell2016reproductive},\cite{posny2014modelling}).
\section{The threshold dynamics of $R_0$}

\subsection{Disease extinction}
\begin{theorem}
Let $ \mathcal{R}_0 $ be defined as \eqref{ec4}, then the disease free periodic solution $( \hat{S}(t),0,0 )$ is asymptotically stable if $ \mathcal{R}_0<1 $ and unstable if $ \mathcal{R}_0>1 $ \label{teo1} .
\end{theorem}
\begin{proof}
We use theorem 2.2 of \cite{wang2008threshold}, and check conditions (A1)-(A7). Conditions (A1)-(A5) are clearly satisfied from definition of $ \mathcal{F} $ and $\mathcal{V}$ given in section 2. We prove only condition (A6) and (A7). Define
$$M(t):= -(\mu + r(t)+ \delta),$$
and let $\Phi_M(t)$ be the monodromy matrix of system
\begin{equation}
 \frac{dz}{dt}=M(t)z. \label{ec5}
\end{equation}

\begin{itemize}
\item [(A6)] $\rho(\Phi_M(LT))<1$. \par
Let $\Psi_M$ be a fundamental matrix for system $ \frac{dz}{dt}=M(t)z $, with $M$ defined as before and $LT$ periodic, the monodromy matrix $\Phi_M(LT)$ is given by $ \Phi_M (LT)= \Psi_M^{-1}(0) \Psi_M(LT) $. The general solution of \eqref{ec5} is $$z(t)=K \exp ( - \int_0^{t} (\mu + r(s)+ \delta)ds ),$$
so $\Psi_M = \exp ( - \int_0^{t} (\mu + r(s)+ \delta)ds ) $ and $ \Psi_M^{-1}= \exp (  \int_0^{t} (\mu + r(s)+ \delta)ds ) $. Note that $\Psi_M^{-1}(0)=1$, so $\Phi_M(LT)=\Psi_M(LT)$
$$ \Phi_M(LT)= \exp ( - \int_0^{LT} (\mu + r(s)+ \delta)ds ). $$
Due to the fact that $\Phi_M(LT)$ is a constant, its eigenvalue is itself and $\rho(\Phi_M(LT))<1$ for $\mu, \delta, r(s)>0$.
\item [(A7)] $\rho ( \Phi_{-V}(LT))<1$. \par 
Solving the system $ \frac{dz}{dt} = -V(t)z $, we arrive to the general solution 
$$ z(t)= c_1 \left( 
\begin{matrix}
\frac{\gamma- \sigma}{\sigma} \\ 1
\end{matrix}
 \right) e^{-(\mu+\sigma )t} +
 c_2 \left( 
\begin{matrix}
0 \\ 1
\end{matrix}
 \right) e^{-(\mu+\gamma )t}, $$
so 
\begin{equation}
\Psi_{-V}(t) = \left( 
\begin{matrix}
\frac{\gamma- \sigma}{\sigma} e^{-(\mu+\sigma )t} & 0 \\
e^{-(\mu+\sigma)t} & e^{-(\mu+\gamma )t}
\end{matrix}
 \right).
\end{equation}
Computing $ \Phi_{-V}(LT)=\Psi_{-V}^{-1} (0) \Psi_{-V}(LT) $ we have 
\begin{equation}
\Phi_{-V}(LT) = \left(
\begin{matrix}
e^{-(\mu+\sigma)LT} & 0 \\ 0 & e^{-(\mu+\gamma)LT}
\end{matrix} \right).
\end{equation}
Clearly, $\rho(\Phi_{-V}(LT))= \max \{ e^{-(\mu+\sigma)LT}, e^{-(\mu+\gamma)LT}  \} <1 $ for $\mu, \gamma, \sigma>0$.
\end{itemize}
\end{proof}

 \begin{note} Due to $\Psi_A$ is a fundamental solution of a periodic system, we can always choose it such that  $\Psi(0)=I$, so the monodromy matrix satisfies $\Phi_A(LT)= \Psi_A(LT)$. This property is used in further analysis. \label{not1}
 \end{note} 
In order to prove the global stability of the disease free periodic solution, we enunciate some useful definitions and some lemmas. \par 

Let $A(t)$ continuous, cooperative, irreducible and $\omega-$periodic $k \times k$ matrix function, and $\Psi_{A}(t) $ the fundamental matrix of system $ x'(t)=A(t)x(t) $. Denote by $ \rho( \Psi_{A}(\omega) ) $ the spectral radius of $ \Psi_{A}(\omega) $ .

\begin{lemma}
Let $p= \frac{1}{\omega} \ln \rho(\Psi_{A}(\omega)) $. Then there exists a positive, $\omega$-periodic function $v(t) $ such that $e^{pt}v(t)$ is a solution of $ x'(t)=A(t)x(t) $ (see proof in Lemma 2.1 of \cite{zhang2007periodic}). \label{lem1}
\end{lemma}

\begin{lemma}
Function $f(I)$ of model \eqref{ec1} satisfy that $f(I) \leq f'(0)I$, $ \forall I \geq 0 $. \label{lem2}
\end{lemma}
\begin{proof}
Using assumptions on function $f$ we have

\begin{equation}
\frac{d}{dI} \left( \frac{f(I)}{I} \right) = \frac{If'(I)-f(I)}{I^{2}} \leq 0,
\end{equation}
so function $f(I)/I$ decreases $ \forall I>0 $ and then $ \frac{f(I)}{I} \leq \lim_{I \rightarrow 0^{+} } \frac{f(I)}{I}=f'(0) $
\end{proof}

\begin{lemma}
Let $(S(t),E(t),I(t))$ a solution of system \eqref{ec2} with initial conditions $(S_0,E_0,I_0) \geq 0$, and $(\hat{S}(t),0,0)$ the infection free periodic solution of \eqref{ec2}, then 
\begin{equation}
\limsup_{t \rightarrow \infty } (S(t)-\hat{S}(t)) \leq 0.
\end{equation} \label{lem3}
\end{lemma}
\begin{proof}
Proof is similar to Lemma 4.1 of \cite{moneim2005use}. $S(t)$ satisfies first equation of system \eqref{ec2},then
\begin{align*}
\frac{dS}{dt} &= \mu N (1-p) - \beta(t)Sf(I)-(\mu+r(t))S+ \delta(N-S-E-I) \\
& \leq N(\mu(1-p)+ \delta) - ( \mu+r(t)+ \delta )S.
\end{align*}
Let $X(t)=S(t)-\hat{S}(t)$, then 
\begin{align*}
\frac{dX}{dt} &= (\mu+r(t)+ \delta)(\hat{S}-S) - \beta(t)Sf(I)- \delta(E+I) \\
&\leq -( \mu+ r(t)+ \delta ) X
\end{align*}
Using Gronwall's inequality $ X(t) \leq X(0) e^{- \int_0^{t}(\mu+r(s)+ \delta)ds},$ so
\begin{align*}
S(t)- \hat{S}(t)\leq& (S(0)- \hat{S}(0))  e^{- \int_0^{t}(\mu+r(s)+ \delta)ds} \\
&= (S(0)- \hat{S}(0)) e^{-(\mu+\delta)t} e^{\int_0^{t}r(s)ds}.
\end{align*}
Applying limit in both sides, we obtain $ \limsup_{t \rightarrow \infty} S(t)- \hat{S}(t) \leq 0 $. 
\end{proof}
Now, we are able to enunciate our theorem for global stability of infection free periodic solution.
\begin{theorem}
The infection free periodic solution $( \hat{S}(t),0,0 )$ of system \eqref{ec2} is globally asymptotically stable if $ \mathcal{R}_0<1 $. \label{teo2}
\end{theorem} 
\begin{proof}
From theorem \eqref{teo1} we have $ (\hat{S}(t),0,0)$ is unstable for $ \mathcal{R}_0>1 $ and asymptotically stable for $ \mathcal{R}_0<1 $, so it is sufficient to prove that any solution $(S(t),E(t),I(t))$ with non-negative initial conditions $(S_0,E_0,I_0)$ approaches to $( \hat{S},0,0 )$.\\
 Let $ \epsilon>0 $, from Lemma \eqref{lem3} we have
$$ \limsup_{t \rightarrow \infty} (S(t)- \hat{S}(t))= \lim_{t \rightarrow \infty}  \sup_{\tau \geq t}\left( S(\tau)- \hat{S}(\tau) \right)= L \leq 0, $$
so there exist a $N>0$ such that for all $t_1 >N$
$$  - \epsilon< \sup_{t \geq t_1}\left( S(t)- \hat{S}(t) \right)-L < \epsilon, $$ 
this implies that $ \sup_{t \geq t_1}( S(t)- \hat{S}(t))< \epsilon+L \leq \epsilon $. Then, from definition of supremum we have for all $t>t_1$
\begin{align*}
S(t)- \hat{S}(t) \leq \sup_{t \geq t_1}( S(t)- \hat{S}(t))< \epsilon.
\end{align*}
Then, we have proved that for all $\epsilon>0$ we can find a $t_1>0$ such that $S(t)< \epsilon + \hat{S}(t)$ for all $t>t_1$. \par
 
Now, using lemma \eqref{lem2} , for $\epsilon>0$ we can find a $t_1>0$ such that for $t>t_1$ 
\begin{align}
\frac{dE}{dt} &= \beta(t)Sf(I)- ( \mu + \sigma )E \nonumber, \\
& \leq \beta(t)S(t)f'(0)I- ( \mu + \sigma )E(t) \\
&< \beta(t)f'(0)(\hat{S}(t)+ \epsilon)I(t)- ( \mu + \sigma )E(t). \label{ec8}
\end{align}
We consider the following  perturbated sub-system:
\begin{align}
\frac{d \bar{E} }{dt} &= \beta(t)f'(0)(\hat{S}+ \epsilon)\bar{I}- ( \mu + \sigma )\bar{E} , \nonumber \\
\frac{d \bar{I}}{dt} &= \sigma \bar{E} - ( \mu + \gamma )\bar{I}, \label{ec7}
\end{align}

which can be rewritten as $$ \left(\frac{d \bar{E}}{dt}, \frac{d \bar{I}}{dt} \right)^{T} = (F(t)-V(t))(\bar{E}, \bar{I})^{T} + \epsilon H(t) (\bar{E}, \bar{I})^{T},  $$
with $F(t),V(t)$ defined in \eqref{ec6} and 
\begin{equation}
H(t) = \left( \begin{matrix}
0 & \beta(t)f'(0) \\0 &0 
\end{matrix} \right)  . \label{ec15}
\end{equation}
Matrix $(F-V+ \epsilon H)(t)$ is  LT-periodic, cooperative, irreducible and continuous. Using lemma  \eqref{lem1}, if $q=  \frac{1}{LT} \ln \rho ( \Psi_{F-V+ \epsilon H} (LT) ) $ then there exist a positive and LT-periodic function $v(t)=(v_1(t),v_2(t))^{T}$ such that $e^{qt}v(t)$ is solution of system \eqref{ec7}. Note that for all $k>0$, function $ k e^{q(t-t_i)} v(t-t_i) $ is also a solution of system \eqref{ec7} with initial condition $kv(0)$ at $ t=t_i $ .\par 
Choose a $ \bar{t}>t_1 $ and $ \alpha_1>0$ such that $(E(\bar{t}),I(\bar{t}))^{T} \leq \alpha_1v(0)$, then from \eqref{ec8}
$$ \left( \frac{dE}{dt}, \frac{dI}{dt} \right)^{T} \leq (F-V)(E,I)^{T}+ \epsilon H(E,I)^{T}, $$

and using comparison principle (see for instance \cite{smith1995theory} theorem B.1), we have $ (E(t),I(t))^{T} \leq \alpha_1 e^{q(t- \bar{t})}v(t- \bar{t}) $ for all $t > \bar{t}$. \par 
From theorem 2.2 of \cite{wang2008threshold}, $ \mathcal{R}_0<1 $ iff $\rho ( \Phi_{F-V}(LT)<1$. By the continuity of the spectrum for matrices (see \cite{kato2013perturbation}, Section II.5.8 ) we can choose $ \epsilon>0 $ small enough that $\rho ( \Phi_{F-V+ \epsilon H} (LT)<1$ and then $q<0$ (see note \eqref{not1} ). So, using positivity of solutions and comparison: 
$$ 0 \leq \lim_{t \rightarrow \infty} E(t) \leq \lim_{t \rightarrow \infty} \alpha_1 e^{q(t- \bar{t})}v_1(t- \bar{t})=0. $$
And similarly for I. We obtain 
\begin{align}
 \lim_{t \rightarrow \infty} E(t)&=0 \nonumber\\
 \lim_{t\rightarrow \infty } I(t)&=0. \label{ec9}
\end{align} 

We need only prove that $S(t)$ approaches to $ \hat{S} $. At infection free solution $ \hat{R}(t)= N- \hat{S}(t),$ where $\hat{R}$ satisfies equation
\begin{equation}
\frac{d \hat{R} }{dt} = \mu Np+r(t) \hat{S}-(\mu+ \delta)\hat{R}.
\end{equation}
So $R(t)=N-S(t)-E(t)-I(t)$ satisfies 
\begin{equation}
\frac{d(R- \hat{R} )}{dt} = r(t)(S-\hat{S})+ \gamma I-(\mu+ \delta)(R- \hat{R}).
\end{equation}
Let $ \epsilon_1>0 $ arbitrary and $r_{max}= \max_{u \in [0,LT]} r(u)$. Due to \eqref{ec9} we can find a $t_2>0$ such that $I(t) < \epsilon_1 $ for $t>t_2$, moreover we can find a $t_3>0$ such that $S(t) \leq \hat{S}(t)+ \epsilon_1$ for $t>t_3$. Then, let $t_4= \max \{t_2, t_3 \}
$, we have for $t>t_4$
$$ \frac{d(R- \hat{R})}{dt} \leq (r_{max}+ \gamma) \epsilon_1-(\mu+ \delta)(R- \hat{R}). $$
Multiplying in both sides by $e^{(\mu+ \delta)}t$ and integrating from $t_4$ to $t$ we obtain
\begin{equation}
(R- \hat{R}) \leq (R- \hat{R})(t_4)e^{-(\mu+ \delta)(t-t_4)} + \frac{\epsilon_1 ( r_{max} + \gamma )}{\mu+ \delta} (1-e^{- (\mu+ \delta)(t-t_4)}).
\end{equation}
So, $ \limsup_{t \rightarrow \infty }(R- \hat{R})(t) \leq \frac{\epsilon_1 ( r_{max} + \gamma )}{\mu+ \delta} $, where $\frac{\epsilon_1 ( r_{max} + \gamma )}{\mu+ \delta}$ is arbitrarily small. Then $ \limsup_{t \rightarrow \infty} (R- \hat{R})(t) \leq 0 $ and using similar arguments to used for $S$ for $ \epsilon_3>0 $ we can find a $t_5>0$ with $R(t) \leq \hat{R}(t) + \frac{\epsilon_3}{2} $ for $t>t_5$. Also, from \eqref{ec9} we can find $t_6>0$ with $E(t)+I(t)<  \frac{\epsilon_3}{2} $ for $t>t_6$, so for $t> \max \{ t_5,t_6 \}$ we have 

\begin{align*}
S(t)&= N-E(t)-I(t)-R(t) \\
& \geq N- \hat{R}(t)- \epsilon_2 = \hat{S}(t)- \epsilon_2.
\end{align*}
Or equivalently, $S(t)- \hat{S}(t) \geq - \epsilon_2$, with $ \epsilon $ arbitrarily small and this implies that $ \liminf_{t \rightarrow \infty} (S- \hat{S})(t) \geq 0 $. We conclude by comparison and using lemma \eqref{lem3} that $ \lim_{t \rightarrow \infty} S(t)= \hat{S}(t) $ completing the proof.

\end{proof}

Theorem \eqref{teo2} shows that disease will completely die as long as $ \mathcal{R}_0<1 $. So, reducing and keeping $\mathcal{R}_0$ below the unity would be sufficient to eradicate infection, even in a periodic environment and a general incidence rate \par

\subsection{Disease persistence}

Uniform persistence is an important concept in population dynamics, since
it characterizes the long-term survival of some or all interacting species in an
ecosystem \cite{zhao2013dynamical}.  \par 
In this section we consider the dynamics of the periodic model when $\mathcal{R}_0>1$. We will show that actually, $ \mathcal{R}_0 $ is a threshold parameter for the extinction and the uniform persistence of the disease. The results are inspired by  \cite{bai2012global}, \cite{posny2014modelling}, \cite{yang2015global} and \cite{zhang2007periodic}. \par

Let $P:X \rightarrow X$ be the Poincar\'e map associated with system \eqref{ec2}, that is 
$$P(x_0)= \phi(LT,x_0), \quad \forall x_0 \in X,$$
where $X$ is defined in \eqref{ec10} and $\phi(t,x_0)$ is the unique solution of system \eqref{ec2} with $\phi(0,x_0)=x_0$. We define the following sets:
$$ X_0:= \{ (S,E,I) \in X : E >0, I>0 \}, \quad \partial X_0 := X \backslash X_0. $$
Note that $\partial X_0$ is not the boundary of $X_0$, but it is a standard notation of persistence theory. \par 

\begin{lemma}
The set $X_0$ is positively invariant under system \eqref{ec2}. \label{lem4}
\end{lemma}

\begin{proof}
Let $x_0=(S_0,E_0,I_0) \in X_0$, ie $E_0>0, I_0>0$ and $\phi(t,x_0)=(S(t),E(t),I(t))$ the solution of \eqref{ec2} with $\phi(0,x_0)=x_0$. Due to non negativity of solutions and assumptions on  function $ \beta(t) $ and $f(I)$, we have:
$$ \dfrac{dE}{dt} = \beta(t) S f(I) - (\mu + \sigma) E \geq  - (\mu + \sigma) E, \quad \forall t>0. $$
Using comparison theorem (see for instance \cite{smith1995theory} Appendix B.1) we have for all $t>0$:
$$ E (t) \geq K e^{-(\mu + \sigma)t}>0, \quad \text{with } \quad K=E(0)>0. $$
Similarly,
$$ \dfrac{dI}{dt} = \sigma E- (\mu + \gamma) I \geq -(\mu + \gamma)I, $$
so,
$$ I(t) \geq I(0) e^{-(\mu + \gamma)}t>0, \quad \forall t>0.$$
Therefore, $\phi(t,x_0)$ remains on $X_0$ for all $t>0$.
\end{proof}
Set:
 $$ M_\partial := \{ (S_0,E_0,I_0) \in \partial X_0 : P^{m} (S_0,E_0,I_0) \in \partial X_0 , \forall m \geq 0\} .$$
 To use persistence theory developed in \cite{zhao2013dynamical}, we show that
\begin{equation}
M _\partial = \{ (S,0,0): S \geq 0\}. \label{ec11}
\end{equation} 
 Let $x_0=(S_0,0,0) \in X$ and $(S(t),E(t),I(t))$ the solution that passes trough that initial condition. We have that $\phi(t,x_0)=(S_1(t),0,0)$, with $S_1(t)$  solution of  \eqref{ec12} and $S_1(0)=S_0$ is a solution that satisfies the initial condition.  By uniqueness of solutions we have $E(t)=0=I(t)$ $\forall t \geq 0$, so $x_0$ lives on $M_\partial$. \par 
 Now, if $x_0 \in M_\partial $ we want $x_0 =(S_0,0,0)$. We prove an equivalent sentence: if $x_0 \in \partial X_0 \backslash \{ (S,0,0) : S \geq 0 \} $ then it does not belong to $M_\partial$.  Consider an initial point $x_0=(S_0,E_0,I_0) \in \partial X_0 \backslash \{ (S,0,0) : S \geq 0 \} $, then $E_0>0, I_0=0$ or $E_0=0, I_0>0$. Suppose $E>0$ and $I_0=0$, then $ \phi(t,x_0)$ holds 
 $$ \frac{dI}{dt}(0)= \sigma E(0)>0. $$
 By continuity of $E(t)$ and sign of derivative of $I$, we have that for small $0<t<<1$, $E(t)>0, I(t)>0$, so, for  $0<t<<1$, $\phi(t,x_0) \in X_0 $. Using invariance of $X_0$ ( Lemma \eqref{lem4} ) we have $\phi(t,x_0) \in X_0$ for all $t>1$. Finally, for an $m>0$ such that $mLT>1$ we have $P^{m}(x_0)=\phi(mLT,x_0) \in X_0$ and this implies  \eqref{ec11}. By discussion in section 2 is clear that there is one fixed point of $P$ in $M_\partial$: $M_0( \hat{S}(0),0,0 )$ ( \cite{nakata2010global} ).   \par 
 Now, we are in position to introduce the following result of uniform  persistence of the disease. 
\begin{theorem}
Let $ \mathcal{R}_0>1, $ then there exists an $ \epsilon>0 $ such that any solution $(S(t),E(t)I(t))$ of \eqref{ec2} with initial values $(S(0),E(0),I(0))\in X_0$ satisfies
\begin{equation}
\liminf_{t \rightarrow \infty} E(t) \geq \epsilon, \quad \liminf_{t \rightarrow \infty} I(t) \geq \epsilon,
\end{equation}
 \label{teo3}
\end{theorem}
 \begin{proof}
 We first prove that $P$ is uniformly persistent (see definition 1.3.2 from \cite{zhao2013dynamical} ) with respect to $(X_0, \partial X_0)$, because this implies that the solution of \eqref{ec2} is uniformly persistent with respect to $(X_0, \partial X_0)$ (  \cite{zhao2013dynamical}, theorem 3.1.1 ). Clearly, $X_0$ is relatively open in $X$, so $ \partial X_0 $ is relatively closed.
  
 Define $$W^{s}:=\{ x_0 \in X_0 : \lim_{m \rightarrow \infty} \Vert P^{ m}(x_0)-M_0 \Vert=0 \},$$
 we show that $W^{s}(M_0) \cap X_0= \emptyset.  $\par 
 By theorem 2.2 of \cite{wang2008threshold}, $ \mathcal{R}_0>1 $ iff $ r(\Psi_{F-V}(LT))>1 $. Choose an $ \eta>0 $ small enough with the property $ \hat{S}(t)- \eta>0, \forall t>0 $ (see appendix \eqref{apb} ). 
 For $\alpha>0$, let us consider the following perturbed equation:
\begin{equation}
\dfrac{d\bar{S}}{dt} = N(\mu(1-p)+ \delta)- 2 \delta \alpha - ( \beta(t)f'(0) \alpha+ \mu + r(t)+ \delta ) \bar{S}. \label{ec13}
\end{equation}
System above admits a unique positive $LT-$periodic solution of the form:
\begin{align}
\hat{S}(t, \alpha) = &e^{-\int_0^{t}( \beta(s) f'(0)\alpha+ \mu + r(s)+ \delta )ds} \\
& \left( \hat{S}(0, \alpha)+ (N\mu(1-p)+N \delta - 2 \delta  \alpha) \int_0^{t} e^{\int_0^{s} ( \beta(\xi) f'(0)\alpha+ \mu + r(\xi)+ \delta )d \xi} ds \right) \label{ec22}
\end{align}

whith $ \hat{S}(t, 0)= \hat{S}(t) $ and which is globally attractive in $ \mathcal{R}_+ $ (see appendix \eqref{apd} ) with 
\begin{align}
\hat{S}(0, \alpha)= \dfrac{ \left(N \mu (1-p)+N \delta - 2 \delta \alpha \right) \int_0^{LT} e^{\int_0^{s}( \beta (\xi) f'(0) \alpha + \mu + r(\xi) + \delta ) d \xi} ds}{e^{\int_0^{LT} ( \beta(s) f'(0) \alpha + \mu + r(s)+ \delta )ds }-1}. \label{ec23}
\end{align}
Since $\hat{S}(0, \alpha)$ is continuous in $\alpha$, then for all $\epsilon>0 $ there is a $\delta>0$ such that for $|\alpha|< \delta$ we have $ |\hat{S}(0, \alpha)- \hat{S}(0,0)| < \epsilon $. Moreover, by continuity of solutions with respect to initial values we can find for all $\bar{\eta}>0$ an $\bar{\epsilon}>0$ such that if $|\hat{S}(0, \alpha)- \hat{S}(0,0)| < \bar{\epsilon}$ then   
$$ |\hat{S}(t, \alpha)- \hat{S}(0,0)|< \bar{\eta}. $$
Therefore, for $\eta$ established before, we can find $\alpha$ small enough such that $ \hat{S}(t, \alpha)> \hat{S}(t)- \eta $, $\forall t>0$. \par 
 Again, by continuity of solutions with respect to initial values, for this small $\alpha>0$, there exists a $ \delta>0 $ such that for all $ (S_0,E_0,I_0) \in X_0 $ with $ \Vert  (S_0,E_0,I_0) -M_0 \Vert\leq \delta  $ then $ \Vert \phi(t,(S_0,E_0,I_0))- \phi(t, M_i) \Vert< \alpha $, $\forall t \in [0,LT]$. \par 
We now claim that 
\begin{equation}
\limsup_{m \rightarrow \infty} \Vert P^{m}(S_0,E_0,I_0)-M_0 \Vert\geq \delta, \quad \forall (S_0,E_0,I_0)\in X_0.
\end{equation}
By contradiction, suppose that  
 
\begin{equation}
\limsup_{m \rightarrow \infty} \Vert P^{m}(S_0,E_0,I_0)-M_i \Vert< \delta, \quad \text{for some} \quad (S_0,E_0,I_0)\in X_0, \quad \text{and} \quad i=1,2.
\end{equation}
Without loss of generality, we can assume that $\Vert P^{m}(S_0,E_0,I_0)-M_0 \Vert< \delta  $ for all $m \geq 0$ (see appendix \eqref{apa}). From above discussion, $ \Vert  \phi(t, P^{m}(S_0,E_0,I_0) )- \phi(t,M_0) \Vert < \alpha  $, $ \forall m \geq 0 $ and $t \in [0,LT]$. \par 
For any $t\geq 0,$ let $t=mLT+t_1$, where $t_1 \in [0,LT)$ and $m=[ \frac{t}{LT} ]$ is the greatest integer less than or equal to $ \frac{t}{LT} $. Then, we get
$$ \phi(t,(S_0,E_0,I_0))- \phi(t,M_0)= \phi(t_1,P^{m}(S_0,E_0,I_0))- \phi(t,M_0)< \alpha. $$    
If we set $\phi(t,(S_0,E_0,I_0))=(S(t),E(t),I(t))$ , then we have $ E(t) \leq \alpha, I(t)\leq \alpha $, $ \forall t \geq 0 $, and from first equation of \eqref{ec2} and lemma \eqref{lem2} we arrive to:

\begin{align}
\dfrac{d S}{dt} \geq  N(\mu(1-p)+ \delta)- 2 \delta \alpha - ( \beta(t)f'(0) \alpha+ \mu + r(t)+ \delta ) \bar{S}. \label{ec14}
\end{align}
Which is exactly the equation in \eqref{ec13}.  Since the unique periodic solution of \eqref{ec13} is globally attractive in $ \mathcal{R}_+ $, we have for $\bar{S}(t, \alpha)$ solution of \eqref{ec13} that $ \lim_{t \rightarrow \infty} \bar{S}(t, \alpha) = \hat{S}(t, \alpha) $. So for $\eta$ given before there exists $T>0$ such that for all $t \geq T$

$$ \vert \bar{S}(t, \alpha)- \hat{S}(t, \alpha) \vert < \eta,$$
or equivalently $\bar{S}(t, \alpha) >\hat{S}(t, \alpha)- \eta$. Moreover, from previous analysis, $ \hat{S}(t, \alpha)- \eta > \hat{S}(t)- \eta $, therefore, using comparison principle on \eqref{ec14} we arrive to 
\begin{equation}
S(t) \geq  \hat{S}(t)- \eta.
\end{equation} 
for $t>T$. \par 
Due to $E(t), I(t)\leq \alpha $, and  $ \alpha $ is fixed small, we can take $\alpha< \epsilon^{*}$ and use  assumption (A5) in introduction and hence (see appendix \eqref{apc} )

\begin{equation}
 \left( \begin{matrix}
\frac{dE}{dt} \\ \frac{dI}{dt}
\end{matrix} \right) \geq (F-V - \eta H- \alpha K)(E,I)^{T} ,  \label{ec16}
\end{equation}

where $F,V,$ are defined in \eqref{ec6}, $H$ is defined as \eqref{ec15} and 
$$ K= \left( \begin{matrix}
0 & - \frac{1}{2} \beta(t) f''(0)[ \hat{S}- \eta ]\\ 0&0 
\end{matrix} \right) .$$

By theorem 2.2 of \cite{wang2008threshold}, we have $\mathcal{R}_0>1$ iff $ \rho(\Phi_{F-V}(LT))>1 $. By continuity of spectrum (see \cite{kato2013perturbation} Section II ) we can find $\alpha, \epsilon$ such that $$ \rho (\Phi_{F-V-\eta H- \alpha K})>1 .$$  Consider the auxiliar system 
\begin{equation*}
 \left( \begin{matrix}
\frac{dE_2}{dt} \\ \frac{dI_2}{dt}
\end{matrix} \right) = (F-V - \eta H- \alpha K)(E_2,I_2)^{T} ,
\end{equation*}
then, using lemma \eqref{lem1} there exist a solution of \eqref{ec16} with the form $e^{p_2 t}v_2(t)$, with $p_2= \frac{1}{LT} \ln ( \rho (\Phi_{F-V- \eta H- \alpha K}(LT)) )>0 $. Choose a $t_{2}> T$, and a small number $\alpha_2>0$ such that $(E_2(t_2),I_2(t_2))^{T} \geq \alpha_2 v_2(0)$. Using comparison principle we get $(E(t),I(t)) \geq \alpha_2 v_2(t- t_2)e^{p_2(t-t_2)},$ which implies $E(t) \rightarrow \infty$ and $I(t) \rightarrow \infty$. This leads a contradiction. \par
The above claim shows that $P$ is weakly uniformly persistent with respect to $(X_0, \partial  X_0). $ Note that $P$ has a global attractor $\hat{S}(0)$ ( \eqref{lem3} ). It follows that $M_0$ is an isolated invariant set in $X$, $W^{s}(M_0) \cap X_0= \emptyset$. Every orbit in $M_\partial$ converges to $M_0$ and $M_0$ is acyclic. By the aciclity theorem on uniform persistence for maps( \cite{zhao2013dynamical} Theorem 1.3.1 and Remark 1.3.1 )is follows that $P$ is uniformly persistent with respect to $(X_0, \partial X_0)$, that is, there exists $ \epsilon>0 $ such that any solution of \eqref{ec2} satisfies $\lim_{t \rightarrow \infty} E(t)\geq \epsilon, $ $\lim_{t \rightarrow \infty} I(t)\geq \epsilon. $
 
 \end{proof}

\section{Numerical simulations}

In this section we provide some numerical simulations to illustrate  the results obtained in our theorems and compare with previous results. \par 
To improve previous models used in references, we use a particular function 

\begin{equation} \label{ec17}
f(I)= \frac{I}{1+aI}, \quad a\geq0,
\end{equation} 
 
 which includes the case $f(I)=I$ used in \cite{bai2012global}. One can check that function \eqref{ec17} satisfies conditions A1)-A5). Using this function, system \eqref{ec2} is rewritten as:
\begin{equation} 
\begin{aligned}
\frac{dS}{dt} &= \mu N(1-p)- \frac{\beta(t)S I}{1+aI} -(\mu + r(t))S+ \delta(N-S-E-I), \\
\frac{dE}{dt} &= \frac{\beta(t)S I}{1+aI}- ( \mu + \sigma )E, \\
\frac{dI}{dt} &= \sigma E - ( \mu + \gamma )I. \label{ec18}
\end{aligned}
\end{equation}
 
Set an initial population $N=2,200,000$ and take time $t$ in years. Suppose $\mu= 0.02$ per year, corresponding to an average human life time of 50 years. Following \cite{bai2012global} take the parameters as follows: $ \sigma=38.5  $ per year, $ \gamma= 100$ per year, $p=0.85, \delta=0$. Choose the periodic transmission as $\beta(t)=\beta_0+0.0002 \cos(2 \pi t)$, with $\beta_0$ the transmission parameter, and the periodic vaccination rate $r(t)=0.1+0.004 \cos(2 \pi t)$. Both functions have period $LT=1$. 

There exists multiple methods for computing the basic reproduction number, via numerical approximations, or finding a positive solution of the equation $ \rho(W(LT,0, \lambda))=1 $ (see theorem 2.1 of \cite{wang2008threshold} ). 
In order to compare with previous works, we firstly approximate the basic reproduction number with its average value $\mathcal{R}_0^{T}$ (used by several authors as \cite{li2014periodic} and \cite{xu2015globalexponential}), so define:

\begin{equation}
R_0^{T} = \rho([F]V^{-1}),
\end{equation}
where $V$ is given by \eqref{ec6} and  $$[F]= \left( \begin{matrix}
 0 & [\beta][\hat{S}] f'(0) \\ 0 & 0
\end{matrix} \right),  $$
with $[\beta], [\hat{S}]$ the average of functions $\beta, \hat{S}$ defined as $[\beta]= \frac{1}{LT} \int_0^{LT} \beta(t)dt,$ $ \hat{S}= \frac{1}{LT} \int_0^{LT} \hat{S}(t) dt$. Computing each average we obtain $R_0^{T}= 549.6702634 \beta_0$, so $R_0^{T}>1$ for $\beta_0 \in (0.001819272510, \infty).$ \par 

Following theorem 2.1 of \cite{wang2008threshold}, to compute $\mathcal{R}_0$, let  $W(t,s, \lambda), t \geq s,$  the evolution operator of the system
 \begin{equation}
 \frac{dw}{dt}= \left(-V(t)+ \frac{F(t)}{\lambda} \right)w \label{ec21}
 \end{equation}
 ie, for each $\lambda \in (0, \infty)$, $\frac{dW(t,s,\lambda)}{dt}=\left(-V(t)+ \frac{F(t)}{\lambda} \right)W (t,s, \lambda), \forall t \geq s$ and $W(s,s, \lambda)=I_{2 \times 2}$, then $\mathcal{R}_0>0$ is the unique solution of $ \rho(W(LT,0,\lambda))=1. $ 

\begin{ex}
To illustrate our results, first fix $\beta_0=0.0018$. Computing $R_0^{T}$ we have $R_0^{T}= 0.9894064741$, which is a first approximation of $R_0$. To solve system \eqref{ec21} numerically, we substitute the terms of expression of $\hat{S}(t)$ in \eqref{ec19}:
\begin{align*}
\hat{S}(t)=&{{\rm e}^{- 0.1200000000\,t- 0.0006366197724\,\sin \left(  6.283185307
\,t \right) }} \\
& \left(  54999.33689+ 6600.0\,\int _{ 0.0}^{t}\!{{\rm e}
^{ 0.1200000000\,s+ 0.0006366197724\,\sin \left(  6.283185307\,s
 \right) }}{ds} \right) 
\end{align*}

 Due to we can not compute analytically the term
 $$\int _{ 0.0}^{t}\!{{\rm e}
^{ 0.1200000000\,s+ 0.0006366197724\,\sin \left(  6.283185307\,s
 \right) }}{ds},$$
   we approach $\hat{S}(t)$ using Taylor expansion around 0 (remember that we want so solve $ \rho(W(LT,0,\lambda))=1,$ where $LT=1$), so even when we can not find an explicit expression for $\hat{S}(t)$, the Taylor expansion is a good way to estimate it in $(0,1)$ .
 
\begin{figure}[H]
\begin{center}
\includegraphics[scale=0.4]{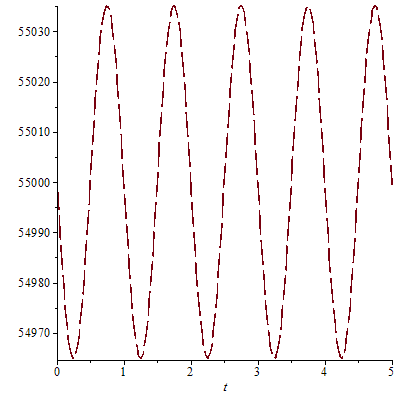}
\includegraphics[scale=0.4]{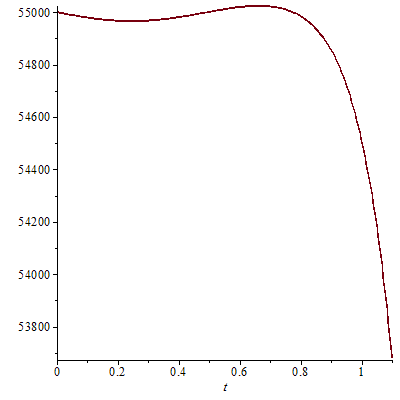}
\end{center}
\caption{$S$  component of infection free periodic solution. Time $t$ is given in years. Left: $\hat{S}(t)$, right: taylor expansion of $\hat{S}$ around $t=0$.}  \label{figure1}
\end{figure}
Setting an initial value $\lambda_0=0.98$ and letting $\lambda_i= \lambda_0 + i (0.0001)$, we solve system \eqref{ec21} numerically for each $\lambda_i$ (using initial conditions $w(0)=(1,0)$ and $w(0)=(0,1)$, to satisfy $W(0,0)=I_{2 \times 2}$ ), and compute $ \rho_1= \rho(W(LT,0,\lambda_i))$ until $\rho_1 \sim 1$ \ . With previous process we arrive to $\rho_1= 1.00120166209265 $ for $\lambda=0.9872$ and $ \rho_1 = 0.997826338969630$ 
for $\lambda=0.9873$, therefore $\mathcal{R}_0 \in (0.9872, 0.9873)$. Using a finer step size $0.0000001$ to have more accuracy, we arrive to $\mathcal{R}_0 \sim 0.9872355<1$. \par 
Set initial values as $S(0)=1,500,000, E(0)=400,000, I(0)=40,000,R(0)=N-(S(0)+E(0)+I(0))$. 

\begin{figure}[H] 
\includegraphics[scale=0.5]{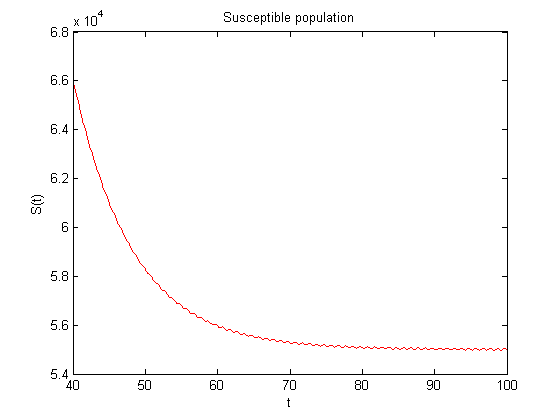}
\includegraphics[scale=0.5]{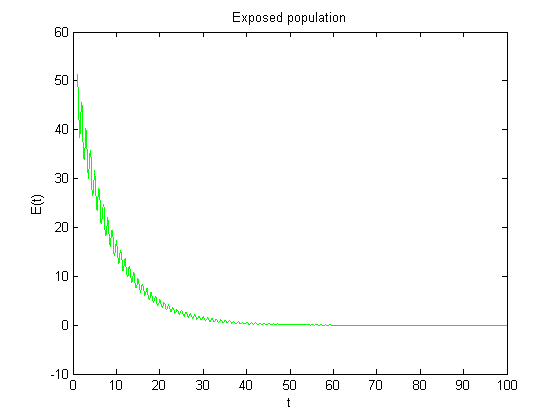}
\includegraphics[scale=0.5]{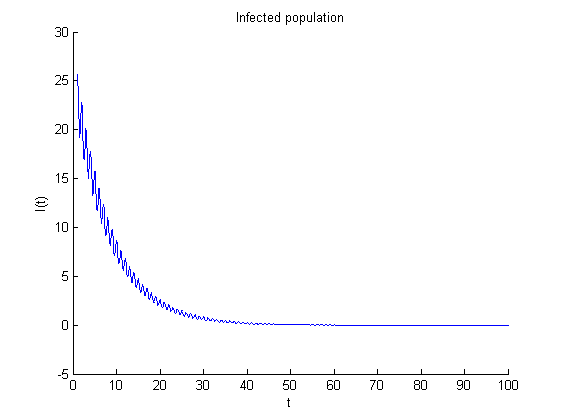}
\includegraphics[scale=0.5]{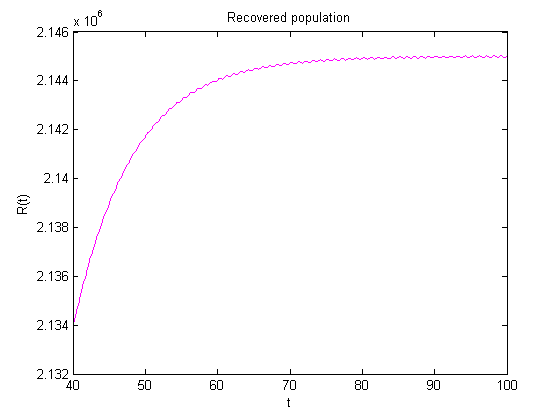}
\caption{Solution of SEIR system when $\mathcal{R}_0<1$. Time $t$ is given in years} \label{figure2}
\end{figure}

 We use Matlab algorithms to graph the solution of system \eqref{ec18} with these initial conditions. Figure \eqref{figure2} shows the results. We can see that $I(t), E(t)$ goes to cero, while $S(t),R(t)$ tend to stabilize, also  $S(t)$ is tending to $\hat{S}(t)$ with values between 54,000 and 56,000 (see figure \eqref{figure1} ), this shows the results obtained in theorem  \eqref{teo2}  . \par 

\end{ex}

\begin{figure}[H] 
\includegraphics[scale=0.5]{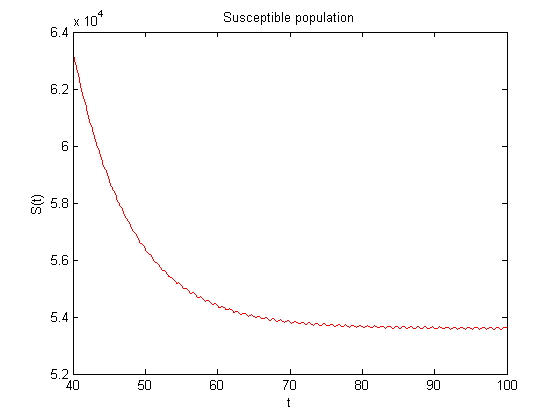}
\includegraphics[scale=0.5]{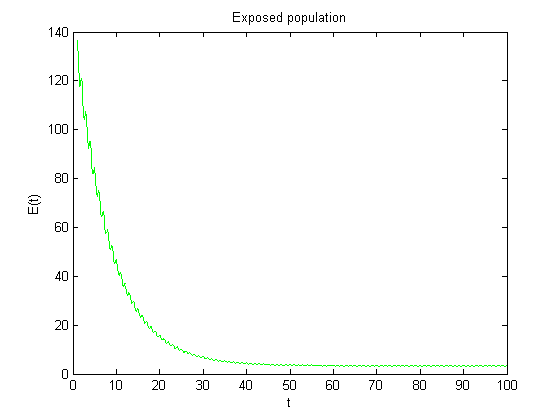}
\includegraphics[scale=0.5]{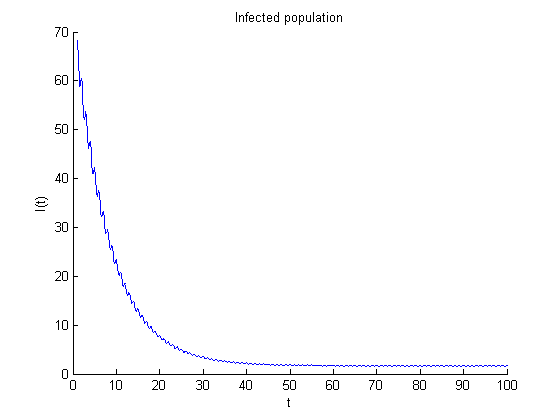}
\includegraphics[scale=0.5]{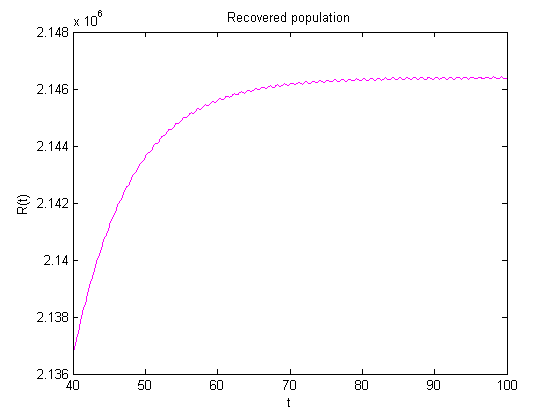}
\caption{Solution of SEIR system when $\mathcal{R}_0>1$. Time $t$ is given in years} \label{figure3}
\end{figure}

\begin{ex}
Now, choose $\beta_0=0.005$. As we can see in figure \eqref{figure3}, the solutions of system remain persistent when $t$ tends to infty, this fact suggest that $\mathcal{R}_0>1$ from theorem \eqref{teo3}. In fact, if we compute the basic reproduction number and its average (using the process described in example 1), $\mathcal{R}_0^{T}=3.298021580 $ and $ \mathcal{R}_0 \in (2.7456,2.7457) $, therefore it is bigger than one. In fact, this shows the results of persistence obtained in theorem \eqref{teo3}.
\end{ex}

\section{Conclusion}
In this paper we presented a model with seasonal fluctuation with a general incidence function that includes the bilinear case studied by \cite{bai2012global}. We proved that $\mathcal{R}_0$ is a threshold parameter for stability and persistence of system, giving also some numerical simulations that show these results. \par 
Several authors (for example \cite{li2014periodic} and \cite{xu2015globalexponential}) define $\mathcal{R}_0^{T}$ as the basic reproduction number, but we can see in numerical simulations that $\mathcal{R}_0^{T}$ is  not equal to $\mathcal{R}_0$ defined by \cite{wang2008threshold}, which is a real threshold parameter for extinction and persistence of disease. \par 
To obtain the estimation of $R_0$ we used a code in Maple, which is based on numerical computing $  \rho_1= \rho(W(LT,0,\lambda_i))$ until $\rho_1 \sim 1$, where $\lambda_i=\lambda_0+ \Delta_\lambda i $, $\Delta_\lambda$ is the step size and the  initial estimation $\lambda_0$ is taken as $R_0^{T}- \epsilon$. The Maple code is available to anyone who wants to use it. \par 


\section{Acknowledgements}
This article was supported in part by Mexican SNI under grant 15284 and 33365.

\appendix

\section{Appendix A:  Assumption used in \eqref{teo3} } \label{apa} 
 Let $f(m):= \Vert P^{m}(S_0,E_0,I_0)-M_i \Vert$. If
\begin{equation}
\limsup_{m \rightarrow \infty} f(m) < \delta, \quad \text{for some} \quad (S_0,E_0,I_0)\in X_0, \quad i=1,2, 
\end{equation}
then we have $L= \lim_{m \rightarrow \infty} \left( \sup_{n \geq m} f(n)  \right)< \delta $. For all $ \epsilon>0 $ there exists a $M_\epsilon>0$ such that if $m \geq M_\epsilon$ then $ -\epsilon< \sup_{n \geq m} f(n)-L< \epsilon $. Particularly, for $\epsilon= \frac{\delta-L}{2}>0$ we have $$ \sup_{n \geq m} f(n)-L< \delta-L ,$$
or equivalently, $\sup_{n \geq m} f(n)< \delta$ for $m \geq M_{\delta-L}$. Moreover, for all $n \geq m$ with $m \geq M_{\delta-L}$ we have $f(n)< \sup_{n \geq m} f(n)< \delta $. Therefore, $\Vert P^{n}(S_0,E_0,I_0)-M_i \Vert< \delta$, $\forall n\geq M_{\delta_L}$. \par 
We can take $(S^{1}_0,E^{1}_0,I^{1}_0)=P^{M_{\delta-L}} (S_0,E_0,I_0)$ as initial condition and therefore,
$$ \Vert P^{n}(S^{1}_0,E^{1}_0,I^{1}_0)-M_i \Vert< \delta, \quad \forall n \geq 0, $$
making our assumption valid.  

So, we can assume without loss of generality that  $\Vert P^{m}(S_0,E_0,I_0)-M_i \Vert< \delta  $ for all $m \geq 0$ . \par 
 
\section{Appendix B} \label{apb} 

 Note that $ \hat{S}(t) $ has a positive minimum $min$(is periodic, positive and continuous, so it is bounded for $t \in [0,LT]$ and then for all $t>0$) and we can choose a $min>\eta>0$,   sufficiently small such that $ \hat{S}(t)-  \eta >0 $. 
 
\section{Appendix C: expression 36} \label{apc}
From system \eqref{ec2} $ \frac{dE}{dt}= \beta(t)Sf(I)-(\mu + \sigma) E,$ with $S(t)> \hat{S}(t)- \eta$ for $t>T$, so 
$$ \dfrac{dE}{dt} \geq \beta(t) ( \hat{S}(t)- \eta )f(I)- (\mu + \sigma)E, \quad \text{for} \quad t>T. $$  
Using assumption (A5) for $f(I)$ and positivity of $ \hat{S}(t)- \eta $, we have also 
$$ f(I) (\hat{S}(t)- \eta) \geq (\hat{S}(t)- \eta)[If'(0)+ \frac{1}{2} I^{2} f''(0) ]. $$
So, 
\begin{align*}
\dfrac{dE}{dt} \geq & \beta(t) (\hat{S}(t)- \eta)[If'(0)+ \frac{1}{2} I^{2} f''(0)]- ( \mu+ \sigma) E ,\\
&= \beta(t)(\hat{S}(t)- \eta) If'(0) + \frac{1}{2} \beta(t)(\hat{S}(t)- \eta) f''(0)I^{2}-(\mu + \sigma) E.
\end{align*}
Due to $0<I<\alpha$ and $f''(0)\leq 0$, then $I^{2}< \alpha I$ and $f''(0)I^{2} \geq f''(0) \alpha I$, applying this we arrive to
\begin{align*}
\dfrac{dE}{dt} & \geq \beta(t)(\hat{S}(t)- \eta) If'(0)+ \frac{1}{2} \beta(t) (\hat{S}(t)- \eta) f''(0) \alpha I, \\
\dfrac{dI}{dt} &= \sigma E-(\mu + \sigma)I.
\end{align*}
This expression can be written as \eqref{ec16}.

\section{Appendix D: Auxiliar from theorem 4} \label{apd}
The system used in the proof of theorem \eqref{teo3} is 
$$ \dfrac{d\bar{S}}{dt} =N (\mu(1-p)+ \delta)-2 \delta \alpha - (\beta(t)f'(0) \alpha+ \mu + r(t)+ \delta)\hat{S} .$$
Solving the equation above, we arrive to the general solution
\begin{align*}
\bar{S}(t) = e^{ -\int_{t_0}^{t}p(s)ds } \left[ \bar{S}(t_0)+ (N(\mu (1-p)+ \delta)-2 \delta \alpha) \int_{t_0}^{t} e^{\int_{t_0}^s ( p( \zeta ) d\zeta ) }ds \right], \label{apd1}
\end{align*}
where $p(s)= \beta(s)f'(0) \alpha+ \mu + r(s)+ \delta$. We shall examine the behaviour of an arbitrary solution $S$. For each $n=0,1,...$ we can use an initial time $\bar{t}_0= t_0+nLT$ with initial point $\bar{S}(\bar{t}_0)$ and see that:

\begin{align*}
S(t_0+(n+1)LT)&= e^{- \int_{t_0+nLT}^{(t_0+nLT)+LT}p(s)ds }\\
& \left[ S(t_0+nLT)+ (N(\mu (1-p)+ \delta)-2 \delta \alpha) \int_{(t_0+nLT)}^{(t_0+nLT)+LT} e^{\int_{t_0+nLT}^s ( p( \zeta ) d\zeta ) ds} \right].
\end{align*}
Due to $p(s)$ is a periodic function, then $$ \int_{t_0+nLT}^{(t_0+nLT)+LT} p(s)ds = \int_{t_{0}}^{t_0+LT}p(s) ds = \int_0^{LT}p(s)ds ,\quad \int_{t_0+nLT}^{s} p(\zeta) d \zeta= \int_{t_0}^{s-nLT}p(\zeta)d\zeta, $$
where $s-nLT\geq t_0$. Then

\begin{align*}
S(t_0+(n+1)LT)&= e^{- \int_{t_0}^{(t_0+LT)}p(s)ds }\\
& \left[ S(t_0+nLT)+ (N(\mu (1-p)+ \delta)-2 \delta \alpha) \int_{(t_0+nLT)}^{(t_0+nLT)+LT} e^{\int_{t_0}^{s-LT} ( p( \zeta ) d\zeta ) }ds \right].
\end{align*}
And using the change of variable $u=s-LT$, then 

\begin{align}
S(t_0+(n+1)LT)&= e^{- \int_{t_0}^{(t_0+LT)}p(s)ds }\\
& \left[ S(t_0+nLT)+ (N(\mu (1-p)+ \delta)-2 \delta \alpha) \int_{t_0}^{t_0+LT} e^{\int_{t_0}^{u} ( p( \zeta ) d\zeta ) }du \right]. \label{apd2}
\end{align}

Equation \eqref{apd2} gives a recursive relationship between the solution at $t_0+nLT$ and after $LT$ times. If we set $S_n=S(t_0+nLT)$,then for each solution $S$ this relationship is described by:
$$S_{n+1}=F(S_n),$$
with $F$ the right side of \eqref{apd2}. If we take $S_i$ and $S_j$, two different values of ${S_n}$, then 
$$|F(S_i)-F(S_j)|=e^{- \int_{T_0}^{t_0+LT} p(s)ds} |S_i-S_j| \leq |S_i-S_j| \leq e^{-(\mu+ \delta)LT} |S_i-S_j|.  $$

Then, $F(S)$ is a contracting map and by Banach fixed point theorem $F$ has a unique fixed point $S_i$ such that $S_{i+1}=F(S_i)=S_i$, or equivalently, $S(t_0+iLT)=S(t_0+(i+1)LT)$. This fixed point can be found for any $S$ that is solution of differential equation with arbitrary initial condition $S(t_0)$ at any time $t_0$. The fixed point has the form:
$$ S(t_0^{*})= \dfrac{( N(\mu(1-p)+ \delta)-2 \delta\alpha) \int_{t_0^{*}}^{t_0^{*}+LT} \left( e^{\int_{{t_0}^{*}} ^{u}p(s)ds} \right) du}{e^{\int_{0}^{LT} p(s)ds }-1}. $$
So, define the function 
$$   S^{*}(t)= \dfrac{( N(\mu(1-p)+ \delta)-2 \delta\alpha) \int_{t}^{t+LT} \left( e^{\int_{t} ^{u}p(s)ds} \right) du}{e^{\int_{0}^{LT} p(s)ds }-1}. $$
$S^{*}$ is a periodic function with period $LT$ and is continuously differentiable with respect to $t$. One can check (by computing the derivative) that $S^{*}(t)$ is a solution of differential equation, so by existence and uniqueness of solutions it can be rewritten as \eqref{ec22} with initial condition \eqref{ec23}.
\bibliographystyle{plain}
\bibliography{references}
 \end{document}